\documentclass[12pt,letterpaper,titlepage]{amsart}
\usepackage{amsmath, amssymb, amsthm, amsfonts,amscd,xr}
\usepackage{enumerate}

\theoremstyle{plain}
\newtheorem{theorem}{Theorem}[section]

\newtheorem{cor}[theorem]{Corollary}

\newtheorem{prop}[theorem]{Proposition}
\newtheorem{lemma}[theorem]{Lemma}
\theoremstyle{definition}
\newtheorem{definition}[theorem]{Definition}

\newtheorem{rmk}[theorem]{Remark}
\numberwithin{equation}{section}
\newtheorem*{theoremA*}{Theorem A}
\newtheorem*{theoremB*}{Theorem B}
\newtheorem*{theoremm1*}{Theorem A'}
\newtheorem*{theoremC*}{Theorem C}
\newtheorem*{theoremD*}{Theorem D}
\newtheorem*{theoremE*}{Theorem E}
\newtheorem*{theoremF*}{Theorem F}
\newtheorem*{theoremE2*}{Theorem E2}
\newtheorem*{theoremE3*}{Theorem E3}

\newcommand{\C}{\mathbb{C}}

\newcommand{\R}{\mathbb{R}}
\newcommand{\N}{\mathbb{N}}

\newcommand{\End}{\operatorname{End}}
\newcommand{\Exp}{\operatorname{Exp}}

\newcommand{\Ad}{\operatorname{Ad}}
\newcommand{\ad}{\operatorname{ad}}

\newcommand{\vol}{\operatorname{vol}}

\usepackage[usenames]{color}

\def\gf{\mathfrak{g}}
\def\hf{\mathfrak{h}}
\def\kf{\mathfrak{k}}
\def\lf{\mathfrak{l}}

\def\pf{\mathfrak{p}}
\def\qf{\mathfrak{q}}

\def\uf{\mathfrak{u}}
\def\vf{\mathfrak{v}}
\def\zf{\mathfrak{z}}
\def\la{\langle}
\def\ra{\rangle}
\def\1{{\bf1}}

\def\U{\mathcal{U}}

\def\oline{\overline}

\title[homogeneous spaces]
{Vanishing at infinity on homogeneous spaces of reductive type}
\begin{document}
\date{December 7, 2014}

\begin{abstract}
Let $G$ be a real reductive group and $Z=G/H$ a unimodular 
homogeneous $G$ space. The space $Z$ is said to satisfy VAI if all smooth vectors in
the Banach representations $L^{p}(Z)$ vanish at infinity, $1\leq p<\infty$. 
For $H$ connected we show that $Z$ satisfies VAI
if and only if it is of reductive type.
\end{abstract}

\author{Bernhard Kr\"{o}tz, Eitan Sayag and Henrik Schlichtkrull}
\thanks{The second author was partially supported by ISF grant N. 1138/10}
\maketitle
\section{Introduction}

In many applications of harmonic analysis of Lie groups it
is important to study the decay of functions on the group. 
For example for a simple Lie group $G$, the fundamental discovery of Howe and Moore 
(\cite{H-M}, Thm.\ 5.1), that the matrix coefficients of non-trivial irreducible 
unitary representations vanish at infinity, is often seen to play an important role.
In a more general context it is of interest to study matrix coefficients
formed by a smooth vector and a distribution vector. 
If the distribution vector is fixed by some closed subgroup $H$ of $G$, 
these generalized matrix coefficients will be smooth functions
on the quotient manifold $G/H$. This leads to the question which is studied
in the present paper, the decay of smooth functions on homogeneous spaces.
More precisely, we are concerned with the decay of smooth $L^p$-functions
on $G/H$.

Let $G$ be a real Lie group and $H \subset G$ a closed
subgroup. Consider the homogeneous space $Z=G/H$ and assume that it
is unimodular, that is, it carries a $G$-invariant measure
$\mu_{Z}$. Note that such a measure is unique up to a scalar
multiple.

\par For a Banach representation $(\pi, E)$ of $G$ we denote by
$E^{\infty}$ the space of smooth vectors.
In the special case of the left regular representation of
$G$ on $E=L^{p}(Z)$ with $1 \leq p <\infty$,
it follows from the local Sobolev lemma that $E^{\infty}$
is the space of smooth functions on $Z$, all of whose
derivatives belong to $L^{p}(Z)$
(see \cite{Poul}, Thm.~5.1). Let $C^{\infty}_{0}(Z)$ be the space of
smooth functions on $Z$ that vanish at infinity.
Motivated by the decay of eigenfunctions on symmetric spaces
(\cite{RS}), the following definition was taken in \cite{KS}:

\begin{definition}
We say $Z$ has the property {\it VAI} (vanishing at infinity)  if for all $1
\leq p < \infty$ we have $$L^{p}(Z)^{\infty} \subset
C^{\infty}_{0}(Z).$$
\end{definition}

By \cite{Poul} Lemma 5.1, 
$Z=G$ has the VAI property for $G$ unimodular
and $H=\{\1\}$. The main result of \cite{KS} establishes
that all reductive symmetric spaces admit VAI.
On the other hand, it is easy to find examples of homogeneous spaces
without this property. For example, it is clear that a non-compact homogeneous
space with finite volume cannot have VAI.

The main result of this article is
as follows.

\begin{theorem}\label{th=1} Let $G$ be a connected real reductive group
and $H\subset G$ a closed connected subgroup such that $Z=G/H$ is unimodular
and of algebraic type.
Then {\rm VAI} holds for $Z$ if and only if it is of reductive type.
\end{theorem}

Here we recall the following definitions, in which $G$ is a real reductive 
group (see \cite{Wallach} for this notion), and for which 
we let $\Ad$ denote the adjoint
representation of $G$ on the Lie algebra $\gf$.

\begin{definition}\label{defi types}
Let $H\subset G$ be a closed connected subgroup.
\begin{enumerate}
\item\label{red type} We say that $H$ is a {\it reductive subgroup} and that
$Z$ is of {\it reductive type}, if
$H$ is real reductive and the representation $\Ad$ of
$H$ on $\gf$ is completely reducible.
\item\label{alg type} We say that $H$ is an {\it algebraic subgroup} and that
$Z$ is of {\it algebraic type} if
$\Ad(H)$ is the connected component of
an algebraic subgroup of $\Ad(G)$.
\end{enumerate}
\end{definition}

In Theorem~\ref{th=1} the implication `{\it only if\/}'
is valid without the assumption of algebraicity, and we do
not know whether `{\it if\/}' is also valid without this assumption.
Note that both (\ref{red type}) and (\ref{alg type}) are
fulfilled when $H$ is semisimple.
Note also that $Z$ is unimodular when it is of reductive type. 

If $Z$ is of reductive and algebraic type and $B\subset G$ is a compact ball, then we show 
in Section \ref{S:sort} (see also \cite{LM}) that 
$$ \inf_{z\in Z} \vol_Z(Bz) >0\, .$$
In view of the invariant Sobolev lemma of Bernstein (see Lemma \ref{Blemma}) this 
readily implies that $Z$ has VAI.

The converse implication is established in
Proposition \ref{non-red}. As a consequence of the proof it is seen
that in the non-reductive case the volume of the above mentioned sets $Bz$
can be made arbitrarily small by letting $z$ tend to infinity in a
suitable direction (see (\ref{limit zero})).

\smallskip{\bf Acknowledgement} We are grateful to an anonymous referee for comments
which have lead to substantial improvements of the paper.

\section{Notation}\label{inv meas}

Throughout $G$ is a connected real reductive group
and $H\subset
G$ is a closed connected subgroup such that $Z:=G/H$ is unimodular.
We write $\mu_Z$ for a fixed $G$-invariant measure and $\vol_Z$ for the 
corresponding volume function.  

\par Let $\gf$ be the Lie algebra of $G$.
We fix a Cartan involution $\theta$ of $G$. 
The derived involution $\gf \to \gf$ will
also be called $\theta$. 
The fixed point set of $\theta$ is a maximal
compact subgroup $K$ of $G$ whose Lie algebra will be denoted $\kf$.
Let $\pf$ denote the $-1$-eigenspace of $\theta$
on $\gf$, then $\gf=\kf\oplus\pf$.
Let $\kappa$ be a non-degenerate invariant symmetric
bilinear form on $\gf$ such that
\begin{equation*}
\kappa|_{\pf}>0,\quad
\kappa|_{\mathfrak{k}}<0,\quad
\mathfrak{k}\perp \pf.
\end{equation*}

Having chosen $\kappa$ we define an inner product on $\gf$ by
$$\la X, Y\ra = -\kappa (\theta(X), Y).$$ We denote by
$\hf$ the Lie algebra of $H$ and by
$\qf$ be its orthogonal complement in $\gf$.

\begin{lemma}\label{exist theta}
The space $Z$ is of reductive type if and only if there exists a 
Cartan involution $\theta$ of $G$ which preserves $H$. With such a choice
we have $[\hf, \qf]\subset \qf$.
\end{lemma}

\begin{proof}
See \cite{Helgason} Exercise~VI~A8 or \cite{Wolf} Thm.~12.1.4.
The last statement follows easily.
\end{proof}

\begin{rmk}\label{remark symmetric}
Let $Z$ be of reductive type and choose $\theta$ and $\kappa$ as above.
Then  $[\qf,\qf] \subset \hf$ if and
only if the pair $(\gf,\hf)$ is symmetric, that is, if and only if
$$\hf=\{X\in\gf\mid \sigma(X)=X\}$$ for an involution $\sigma$ of $\gf$.
When $\gf$ is semisimple it then follows that
$$\qf=\{X\in\gf\mid \sigma(X)=-X\}.$$
\end{rmk}

\section{VAI versus volume growth}

For a compact set $B\subset G$ we shall consider the volume function
$$F_B: G \to \R_{\ge 0} , \ \ g\mapsto {\rm vol}_Z(Bg\cdot z_0)\, .$$
For that we recall some results from \cite{B}.
By a {\it ball} in $G$ we will understand a compact symmetric neighborhood of $\1$. 
A continuous function $w:G \to \R_+$ is called a {\it weight} provided that for all balls 
$B\subset G$ there exists a
constant $C_B>0$ such that $w(xg) \leq Cw(g)$ holds for all $x\in B$ and $g\in G$ 
(see \cite{B}). 
Two weights  $G\to \R_+$ are called {\it comparable} if their mutual ratio
is bounded from above and below by positive constants.

Let $Z(G)$ denote the center of $G$.

\begin{lemma} \label{w-functions}
Fix a ball $B\subset G$. Then
\begin{enumerate} 
\item $F_B$ is a weight.
\item If $B'\subset G$ is another ball, then $F_B$ is comparable to $F_B'$. 
\item  $F_B$ factors to a continuous 
function on $\Ad(G) \simeq G/Z(G)$.
\end{enumerate}
\end{lemma}

\begin{proof} The last statement is easy. For the others,
see \cite{B} p.~683, Lemma-Definition. In the proof
it is shown that $m_Z:=F_B^{-1}\mu_Z$ is a so-called standard measure.
\end{proof}

\par Let $1\leq p < \infty$. 
For every $k\in \N$ we let $\|\cdot\|_{p, k}$ be a $k$-th Sobolev norm of $\|\cdot\|_p$, the $L^p$-norm on 
$L^p(Z)$ (see \cite{BK}, Section 2).  Note that the collection $\{\|\cdot\|_{p, k}: k\in \N\}$ determines the Fr\'echet topology on $L^p(Z)^\infty$.

For a subset $\Omega\subset Z$ we write $\|\cdot \|_{p, k, \Omega}$ for 
the semi-norm on $L^p(Z)^\infty$ which is 
obtained by integrating the derivatives over  $\Omega$.

In this context we recall the invariant Sobolev lemma of Bernstein:

\begin{lemma}\label{Blemma}
Fix $k> \frac{\dim G}{p}$. 
Then  for every ball $B$ 
there is a constant $C_B>0$ such that 
\begin{equation*} |f(z)| \leq C_B {\rm vol}_Z(Bz)^{-\frac{1}{p}}
%{\bf v}_B(z)^{-\frac{1}{p}} 
\|f\|_{p, k, Bz} \qquad (z\in Z)
\end{equation*}
for all smooth functions $f$ on $Z$. 
\end{lemma}

\begin{proof} See \cite{B}, ``Key lemma'' on p. 686, and note that $m_Z:=F_B^{-1}\mu_Z$ is a standard measure.
The cited lemma has $p=2$, but its proof is valid for $1\leq p < \infty$ as well.
\end{proof}

For $v\in\U(\gf)$ and $f \in L^{p}(Z)^{\infty}$,  as $L_vf$
belongs to $L^p(Z)$, its norm over $Bz$ will be arbitrarily 
small for $z$ outside a sufficiently large compact set. Hence,  for $f\in L^p(Z)^\infty$ with $1\leq p<\infty$  
we obtain that 
$$ \lim_{z\to \infty} \|f\|_{p,k, Bz}= 0\, .$$
Hence we have shown that:

\begin{prop} \label{VAI-vol}
If $\inf_{g\in G} F_B(g)>0$ for some ball $B$,
then ${\rm VAI}$ holds.
\end{prop}

When $\vf\subset\gf$ is a complementary subspace to $\hf$, 
$$\gf= \vf\oplus \hf,$$
we let $\pi_\vf$ 
denote the projection $\gf\to\vf$ along $\hf$, and accordingly identify
$\vf\simeq \gf/\hf$ with the tangent space $T_{z_0}Z$ of $Z$ at $z_0$. 
Given $g\in G$ we further note that the differential of the left multiplication
$\tau_g: Z\to Z$ by $g$ provides an isomorphism 
\begin{equation}\label{tg sp}
d\tau_g: T_{z_0}Z=\vf \xrightarrow{\sim} T_{g\cdot z_0}Z.
\end{equation}

\section{Algebraic lower bound of the volume function}
We know from
Lemma \ref{w-functions} that $F_B$ factors 
through the adjoint representation $G\to \Ad(G)$.
Since $G$ is real reductive, the induced map from
$G/H$ to $\Ad(G)/\Ad(H)$ 
is a finite covering, hence preserves the invariant measure 
up to normalization. It follows that the factored map 
on $\Ad(G)$ agrees with the corresponding map 
$F_{\Ad(B)}$ for $\Ad(G)/\Ad(H)$. 
In order to study $F_B$ we may hence assume that 
$G$ is adjoint. In particular, we can assume it is an 
algebraic group. 

For the following lemma we assume (in addition to $G$ being real reductive)
that $G/H$ is 
{\it real algebraic}. By this we mean
that $G$ and $H$ are the connected components
of the real points of a pair
$G_\C\supset H_\C$ of complex algebraic groups, and thus
$$Z=G/H\subset Z_\C:= G_\C/H_\C.$$

\begin{lemma}\label{algebraic bound}
Assume $G/H$ real algebraic and let $B\subset G$ be a ball.
Then there exists a left $K$-invariant and right $H$-invariant
algebraic function $F$
on $Z$ such that $F(\1)>0$ and
\begin{equation}\label{F<FB}
0\leq F(g)\le F_B(g)^2
\end{equation}
for all $g\in G$.
\end{lemma}

\begin{proof} We need a few geometric preparations.  
Let $\vf\subset \gf$ be a vector complement with a basis
$Y_1, \ldots, Y_n$ consisting of semisimple elements. 
We define a map $\Exp:\vf \to G$ by 
$$\Exp(\sum_{j=1}^n t_j Y_j):=\exp(t_1 Y_1) \cdot \ldots\cdot \exp(t_n Y_n)\, ,$$
and for $g\in G$ we then consider the smooth map 
$$\Phi_g: \vf \to Z, \ \ Y\mapsto \Exp(Y)g\cdot z_0\, .$$
Set $y_i:=\exp(tY_i)$. 

If for each $Y\in\vf$ 
we identify $T_{\Phi_g(Y)}Z$ with $\vf$ as in (\ref{tg sp}), we see 
that the differential of $\Phi_g$ at $Y$ is given by 
\begin{equation}\label{dPhi}
d\Phi_g(Y) (Y') = 
\pi_\vf\big(\Ad(g)^{-1} 
\sum_{j=1}^n t_j' \Ad(y_{j+1}\cdot \ldots\cdot y_{n})^{-1}Y_j\big)
\end{equation}
for  $Y'=\sum_{j=1}^n t_j' Y_j$.
In particular $\Phi_\1$ defines a local diffeomorphism at $Y=0$. 
We are concerned with the cardinality of the fibers
$\Phi_g^{-1}(z)\subset \vf$ at generic elements $z\in Z$
and for generic $g\in G$.

\begin{lemma} There exists $N\in \N$ such that the generic fibers
of $\Phi_g$ are bounded by $N$ for generic elements $g\in G$. 
\end{lemma}

\begin{proof} We recall the following result from algebraic geometry 
(see \cite{G}, Prop.~15.5.1(i)): Let $Z_1, Z_2, Z_3$ be complex 
irreducible algebraic varieties with $\dim Z_1 =\dim Z_3$ and further 
$$f: Z_1 \times Z_2 \to Z_3$$
be an algebraic map, such that for one $z_2'\in Z_2$ the map 
$f(\cdot, z_2')$ is dominant. Then there exists an $N\in \N$ such that 
the generic fibers of $f(\cdot,  z_2)$ are bounded by $N$ for all
generic  $z_2\in Z_2$.  

\par We apply this to 
$Z_1=\exp(\C Y_1) \times\ldots\times \exp(\C Y_n)$, 
$Z_2=G_\C$, $Z_3=Z_\C$, and the map 
$$f((z_1, \ldots, z_n), g):= z_1\cdot\ldots\cdot z_n g\cdot z_0\, .$$
Observe that $f$ is defined over $\R$. The assertion follows. 
\end{proof}

We can now complete the proof of Lemma \ref{algebraic bound}.
Fix an open compact neighborhood $V\subset \vf$ of zero
with $\Exp(V)\subset B$ and for which $\Phi_\1$ restricts
to a diffeomorphism onto its image. Set 
$\phi_g:=\Phi_g|_V$. 
It follows from our formula (\ref{dPhi}) for the differential, that 
the Jacobian 
$$J_g(Y):= \det d\phi_g(Y) \qquad (g\in G, Y\in V)$$
depends algebraically on $g$. 
If $\Omega$ is the $G$-invariant differential form of $Z$ we let 
$\omega_g$ be its pull-back to $V$ and define a function 
$$f_V(g):=\int_V \omega_g\qquad (g\in G)\, .$$
Then it is clear that $f_V$ is a polynomial function on $G$ with 
$f_V(\1)>0$. 

It follows from the uniform fiber bound that 
$$ |f_V(g)|\leq  N \cdot F_B(g)  $$
for $g\in G$ generic, and hence for all $g\in G$
by continuity.
Hence $F_V:=f_V^2/N^2$ is a non-negative 
algebraic function which is dominated by $F_B^2$.

\par It follows from Lemma \ref{w-functions} that 
we can assume in addition that the ball $B$ is right 
$K$-invariant, that is,
\begin{equation}\label{ballK} BK=B\, .\end{equation}
Then the volume function $F_B$ is left $K$-invariant, and hence 
the average of $F_V$ over $K$ from the left is 
algebraic and satisfies (\ref{F<FB}).
\end{proof}

\begin{cor} \label{matrix coefficient} Let $G/H$ be of
algebraic type (see Definition \ref{defi types}(\ref{alg type}))
and let $B\subset G$ be a ball.
There is a finite dimensional representation $(\pi, W)$ of $G$ with  a cyclic $K$-fixed vector
$v_K\in W$ and a cyclic $H$-fixed vector $v_H\in W$ such that 
$\la v_H, v_K\ra> 0$ and 
\begin{equation} \label{mc} 0\le\la \pi(g)v_H, v_K\ra
\le F_B(g)^2 \qquad (g\in G)\, .\end{equation}
Here $\la\cdot,\cdot\ra$ is an inner product on $W$ which is $\theta$-covariant:
$\la \pi(g)v, w\ra = \la v, \pi(\theta(g))^{-1} w\ra$ for $g\in G$ and $v,w\in W$.
\end{cor}

\begin{proof} It follows from the remark in the beginning
of this section that we may assume $G/H$ is real algebraic. 
With the right action
the algebraic function $F$ of Lemma \ref{algebraic bound}
generates a finite dimensional representation $W$ in which $v_H=F$
is $H$-fixed and cyclic. Moreover, evaluation at $\1$ is a
$K$-fixed cyclic vector for the dual representation. Finally, the inner product $\la\cdot,\cdot\ra$ 
exists since $\theta$ is a Cartan involution, and with that we obtain $v_K$ and 
$F(g)=\la\pi(g)v_H, v_K\ra$.
\end{proof}

\section{Reductive Spaces are {\rm VAI}}\label{S:sort}
For $G$ and $H$ both semisimple 
it was shown with analytic methods in \cite{LM} that 
$\inf_{g\in G} F_B(g)>0$.
In this section we give a geometric proof,
which is valid more generally
for spaces which are of both reductive and algebraic type. 
Combined with Proposition   \ref{VAI-vol} 
this completes the proof of the implication `if' of Theorem \ref{th=1}.

\begin{lemma}\label{v>c}
Let $Z=G/H$ be of reductive and algebraic
type and let $B\subset G$ be a ball. Then there
exists a constant $c>0$ such that
\begin{equation} \label{v-bound} 
\vol_Z(Bz) \ge c\end{equation}
for all $z\in Z$. 
\end{lemma}

\begin{proof}
By Lemma \ref{w-functions}
it is no loss of generality to request in addition to (\ref{ballK})
that $B$ has the property:
\begin{equation}\label{thetaB}\theta (B) = B.\end{equation}

As  $Z$ is of reductive type we can apply Lemma \ref{exist theta}
and arrange that $H$ is
$\theta$-stable. 
Then $\theta$ induces an automorphism on $Z$ which is measure preserving. 
Hence (\ref{thetaB}) implies that 
\begin{equation}\label{vBtheta}
F_B(g) = F_B(\theta(g)) \qquad (g\in G).
\end{equation}
 
\par Let $F$ be a matrix coefficient as in
Corollary \ref{matrix coefficient} such that 
$$0\le F(g)\leq F_B(g)^2$$ for all $g\in G$.
Because of (\ref{vBtheta}) we also have
$$0\le F(\theta(g))\leq F_B(g)^2$$ 
for all $g\in G$. Hence it suffices to show
$$\inf_{g\in G} [F(g)+F(\theta(g))]>0.$$

\par We recall the following fact from convex geometry. 
Let $(W_\R, \la \cdot, \cdot\ra)$ 
be an Euclidean vector space and  $C\subset W_\R$ a regular cone, i.e.
$C$ is convex, closed, contains no lines, and has non-empty interior. Let $C^\star\subset W$ be 
the dual cone to $C$. Then $C^\star$ is regular as well. Fix an element $v^\star$ in the interior 
of $C^\star$. Then there exists a constant $c>0$ such that 
\begin{equation} \label{convexity bound}  (\forall v\in C) \ \  \la v^\star, v\ra \geq c \sqrt{\la v, v\ra}\, .\end{equation} 

\par We wish to apply this fact to $F$ and  
the representation $W$ in Corollary \ref{matrix coefficient}.
Note that $W$ has 
a real structure $W_\R$ with $v_K, v_H\in W_\R$.  
As these vectors are cyclic, the closed convex cones
$C_H$ and $C_K$, generated by the $G$-orbit through the rays $\R^+ v_H$ and $\R^+v_K$, respectively,
both have non-empty interior.  As $F$ is non-negative 
we clearly have $C_H \subset C_K^\star$ and 
$C_K \subset C_H^\star$. As $C_K$ is regular, we conclude that $C_H$ is regular as well.  Further 
$v_K$ lies in the interior of $C_K$  (see \cite{HO}, Lemma 2.1.15) and with (\ref{mc}) and 
(\ref{convexity bound})
we obtain a constant $c>0$ such that
\begin{equation} \label{shown1} F(g)\ge c \| \pi(g) v_H\|\end{equation}  
for all $g\in G$. 

For every $X\in \pf$ we let $v_H =v_H^+ +v_H^0+ v_H^-$ 
be the decomposition into positive, fixed and negative eigenvectors for $X$. 
We obtain for $g=\exp X$ that 
$$\|\pi(g)v_H\|^2\geq \|v_H^0\|^2 + \|v_H^+\|^2 $$
and
$$\|\pi(\theta(g))v_H\|^2\geq \|v_H^0\|^2 + \|v_H^-\|^2 .$$
Hence by (\ref{shown1}) 
$$ F(g)+F(\theta(g))\ge c \left(\|v_H^0\|^2 + 
\|v_H^+\|^2+\|v_H^-\|^2\right)^{\frac12}= c\|v_H\|$$
and the lemma is proved.
\end{proof}

\begin{rmk} If $Z=G/H$ is a reductive real spherical space
(in particular, a reductive symmetric space), an upper volume bound 
of exponential  type is also valid. See \cite{vg}.
\end{rmk}

\begin{rmk} For a semisimple symmetric space the wave front lemma,
Theorem 3.1 of \cite{EM}, shows that there exists an open neighborhood 
$V$ of $z_0$, such that $B z$ contains a $G$-translate of $V$ for all $z\in Z$.
This implies (\ref{v-bound}) for this case. 
\end{rmk}

\section{The differential of exp}\label{differentials}

Let $\vf\subset\gf$ be a complementary subspace to $\hf$, and consider the map 
\begin{equation}\label{Phi}
\Phi_g: \vf \to Z,  \ \ Y \mapsto \exp(Y) g\cdot z_0\, .
\end{equation}
The following formula for its differential is well known.

\begin{lemma}\label{jacobian formula}
The differential of
$\Phi_g$ at $Y\in\vf$ is given by
\begin{equation} \label{differential of Phi_X}
d\Phi_g(Y)=d\tau_{\exp(Y)g}\circ\pi_\vf\circ\Ad(g)^{-1}\circ\beta(Y)\circ\iota_\vf
\end{equation}
where
$$\beta(Y)=\frac{\1- e^{-\ad Y}}{\ad Y}\in\End(\gf)$$
for $Y\in\vf$, and  $\iota_\vf: \vf\to\gf$ is the inclusion map.
\end{lemma}

\begin{rmk} \label{remark jacobian formula} In fact
we shall apply the lemma in a more general situation where the complementary subspace
$\vf$ splits in a direct sum of subspaces.
For example if $\vf=\vf_1\oplus\vf_2$ we can replace (\ref{Phi}) 
by
$$\Phi_g: \vf_1\times\vf_2 \to Z,  \ \ (Y_1,Y_2) \mapsto \exp(Y_1)\exp(Y_2) g \cdot z_0\, .$$
Similar to (\ref{differential of Phi_X}) we find in this case for $W=(W_1,W_2)\in\vf$ that
$$d\Phi_g(Y)(W)=d\tau_{\exp(Y_1)\exp(Y_2)g}\pi_\vf \Ad(g)^{-1} (S_{Y,W}),$$
where 
$$S_{Y,W}:=\Ad(\exp(Y_2)^{-1})\beta(Y_1)(W_1) +
\beta(Y_2)(W_2)\in\gf.$$
\end{rmk}

\section{Non-reductive spaces are not VAI}\label{appendix A}

In this section we prove that VAI does not hold on any
homogeneous space $Z=G/H$ of $G$, which is not of reductive type.
We maintain the assumptions in Section \ref{inv meas}
and establish the following result.

\begin{prop} \label{non-red}
Assume that $Z=G/H$ is unimodular
and not of reductive type. Then for all
$1\leq p< \infty$ there exists an unbounded
function $f\in L^p(Z)^\infty$. In particular, {\rm VAI} does not
hold. \end{prop}

\begin{proof} As in Lemma \ref{algebraic bound} the key to the proof is the construction of a 
suitable vector complement $\vf$ to $\hf$ in $\gf$.

Let $\uf_H$ be the largest ideal of $\hf$ which acts by nilpotent morphisms on $\gf$.
As $H$  is not reductive in $G$ we have $\uf_H\neq \{0\}$. Let $L_H<H$ be a Levi-complement 
to $U_H$. According to Borel and Tits (see \cite{BorelTits} or \cite{HU}, Sect.~30.3,  Cor.~A)
we find a parabolic subgroup $Q$ of $G$ with Levi decomposition 
$Q=LU$ such that $L_H \subset L$ and $U_H \subset U$.
Let $\theta$ be a Cartan 
involution of $G$ which fixes $L$ and let $\oline{U}=\theta (U)$. 
We recall that according to Bruhat decomposition,
\begin{equation}\label{Bruhat} \oline{U} \times L \times U \to G, \ \ (\oline u, l , u)\mapsto \oline u l u \end{equation}
is a diffeomorphism onto its Zariski open image. 

Let $X \in \zf(\lf)$ be an element in the center of $\lf$ such that $\ad X|_{\uf}$ has positive spectrum. 
Notice that we cannot have $X\in \hf$, as in that case $\ad X$ would have a positive trace 
on $\hf=\lf_H+\uf_H$, contradicting that $G/H$ is unimodular.
It follows that $a_t\cdot z_0\to\infty$ in $L/L\cap H$ and hence also in $Z$, for $|t|\to\infty$.

\par We now construct a complementary subspace $\uf_X$ to $\uf_H$ as follows.
If $\uf_H=\uf$, then $\uf_X=\{0\}$. Otherwise we choose an
$\ad X$-eigenvector, say $Y_1$, in $\uf\setminus \uf_H$
with largest possible eigenvalue. If $\uf_H+\R Y_1 \subsetneq \uf$
we choose an eigenvector $Y_2\in \uf\setminus (\uf_H+\R Y_1)$
with largest possible eigenvalue. We 
continue this procedure until $Y_1,Y_2,\dots$
span a complementary subspace. This subspace we denote $\uf_X$. 

\par Let $\lf_0=\lf_H^{\perp_\lf}$ denote the orthocomplement of $\lf_H$ in $\lf$. 
Then 
\begin{equation*}%\label{defvf}
\vf=\oline\uf+\lf_0+\uf_X
\end{equation*}
is an $\ad(X)$-stable complement
to $\hf$ in $\gf$.

Before proceeding we note some important consequences of this construction of $\vf$. 
Firstly it follows that
\begin{equation} \label{exp2}
\uf_X \to U/U_H, \ \ Y\mapsto \exp(Y) U_H\end{equation} 
is a diffeomorphism. This boils down to a general property of graded
nilpotent Lie algebras that will be established in Lemma \ref{grad-lemma}.
Secondly the following lemma holds.

\begin{lemma}\label{M_t} 
With $\uf_X$ and $\vf$ defined as above
we have $\sup_{t<0}(M_t)<\infty$ where
$$M_t:=\sup_{W\in \gf, \|W\|=1}  \|\Ad(a_t) \pi_\vf \Ad(a_t)^{-1} W\|.$$
\end{lemma}

\begin{proof}
For $W\in\vf$ we have
$$\Ad(a_t) \pi_\vf \Ad(a_t)^{-1} W=W,$$ 
and for $W\in\lf_H$ we have
$$\Ad(a_t) \pi_\vf \Ad(a_t)^{-1} W=0.$$
Hence we may assume $W\in\uf_H$. We can write $W$ as
a combination of $\ad X$-eigenvectors $Y_\lambda\in\uf$ with eigenvalues $\lambda$.
Then
$$\Ad(a_t)^{-1} W=\sum e^{-\lambda t}Y_\lambda
%=U+\sum (e^{-\lambda t}-1)Y_\lambda
.$$
If $Y_\lambda\in\uf_X$ then
$$\Ad(a_t) \pi_\vf e^{-\lambda t}Y_\lambda=Y_\lambda.$$
Finally if $Y_\lambda$ is not in $\uf_X$, then it is the
sum of an element from $\uf_H$
and some eigenvectors $V_\mu\in\uf_X$. 
Moreover, all these $V_\mu$ 
must have eigenvalues $\mu\geq\lambda$, since otherwise $Y_\lambda$ would have been 
preferred before such a $V_\mu$ in the construction of $\uf_X$.
Thus, 
$$\Ad(a_t) \pi_\vf e^{-\lambda t} Y_\lambda=
\sum_{\mu\ge\lambda} e^{(\mu-\lambda) t} V_\mu$$
which stays bounded for $t\to-\infty$. 
\end{proof}

We now continue with the proof of Proposition  \ref{non-red}.
Let $V_0\subset \lf_0$ be an open neighborhood of $0$ such that 
$V_0\to L/L_H, \ Y\mapsto \exp(Y) L_H$ 
is a diffeomorphism onto its  
image.  It follows that the map 
\begin{equation} \label{exp1} V_0 \times U/U_H \to Q/H, \ \ (Y, uU_H)\mapsto \exp(Y)u \cdot z_0\end{equation}
is a diffeomorphism onto its image. 

Combining (\ref{exp1}) and  (\ref{exp2}) with 
(\ref{Bruhat}) we obtain a diffeomorphism 
\begin{align*} \Phi:\, &\oline{\uf} \times V_0 \times \uf_X\to G/H, \\ 
&(Y^-, Y^0, Y^+)\mapsto \exp(Y^-) \exp(Y^0)\exp(Y^+)\cdot z_0
\end{align*}
onto its image.

\par Further we let 
$V^-$ and $V^+$ be open relatively compact convex neighborhoods 
of $0$ in the vector spaces $\oline\uf$ and $\uf_X$. 
Set $V:=V^- \times V^0 \times V^+$. 

For $t\in \R$ we set $a_t:=\exp(tX)$ and consider the map $\Phi_t: V\to G/H$,
$$\Phi_t( Y):=\exp(Y^-) \exp(Y^0)\exp(Y^+) a_t \cdot z_0$$
where $Y=(Y^-, Y^0, Y^+)\in V$.
It follows that $\Phi_t$ is a diffeomorphism onto its open image for all 
$t\in \R$. We need the following property for which we recall 
the identification (\ref{tg sp}) of the tangent spaces of $Z$ with~$\vf$.

\begin{lemma}\label{key lemma}
There exists a
linear map $L(Y):\vf\to \gf$  such that
\begin{equation}\label{L(Y)}
d\Phi_t(Y)=\Ad(a_t)^{-1}(\1_\vf+   \Ad(a_t) \pi_\vf \Ad(a_t)^{-1} L(Y))
\end{equation}
for all $t\le 0$, and such that $\|L(Y)\|\to 0$ for $Y\to 0$.
\end{lemma}

\begin{proof}
Let $Y=(Y^-, Y^0, Y^+)$ and $X=(X^-, X^0, X^+)$ in $\vf$, 
then it follows from Remark \ref{remark jacobian formula} that
$$d\Phi (Y^-, Y^0, Y^+)(X^-, X^0, X^+)=d\tau_{y^-y^0y^+a_t}(z_0)
\circ
\Ad(a_t)^{-1} (S_{Y,X})$$
where $y^-=\exp(Y^-)$ etc, and where $S_{Y,X}\in\gf$ is the element
$$\Ad(y_0y^+)^{-1}\beta(Y^-)(X^-)+
\Ad(y^+)^{-1}\beta(Y^0)(X^0) +
\beta(Y^+)(X^+).$$
Defining $L(Y)$ by $L(Y)(X)=S_{Y,X}-X$ for $X\in\vf$,
we obtain the expression in (\ref{L(Y)}).
It is easily seen that $\|L(Y)\|\to 0$ for $Y\to 0$.
\end{proof}

Let $J_t=|\det d\Phi_t|$. 
By Lemmas \ref{key lemma} and 
\ref{M_t} there exists a
constant $C>0$ such that the following bound holds
for $V$ sufficiently small: 
\begin{equation} \label{Jac-bound}  J_t(Y) \leq C  e^{t\lambda_X} \qquad (t\leq 0, Y\in V)\, 
\end{equation}  
with $\lambda_X=-\operatorname{trace}\ad_X|_{\oline \uf + \uf_X}$. 
Note that $\lambda_X>0$
since $\uf_H$ is non-trivial.

Fix a function $\psi\in C_c^\infty(V)$ 
with $0\leq \psi\leq 1$ and $\psi(0)=1$. 
For all $t\in \R $ define $\chi_t\in C_c^\infty(Z)$ by
$\chi_t (z)=\psi (\Phi_t^{-1}(z)) $
and set
$$\chi:=\sum_{n\in \N} n\chi_{-n}\, .$$
It is clear 
that $\chi\in C^\infty(Z)$ and that
$\chi$ is unbounded. 
We claim that $\chi \in  L^p(Z)^{\infty}$ for all $1\leq p<\infty$.

It follows from the estimate in (\ref{Jac-bound}) that
for all $1\leq p<\infty$ there exists $C>0$ such that
$\|\chi_t\|_p\leq Ce^{t\lambda_X/p}$ 
for all $t\leq 0$.
Hence $$\chi=\sum_{n\in \N} n\chi_{-n}\in L^p(Z)$$ for all $1\le p<\infty$, and 
it only remains to be seen that also the derivatives of $\chi$ belong to
$L^p(Z)$. 

We first show this for first order derivatives.
Let $W\in\gf$  and consider the derivative
$L(W)\chi_t$. At $z=\Phi_t(Y)$
this is given by
$$L(W)\chi_t(z)=d/ds|_{s=0}\, \chi_t(\exp(sW)ya_tz_0)$$
where $y=\exp(Y)$. For $Y$ in a compact set, we can replace $W$ by its conjugate by
$y$ without loss of generality, and thus we may as well
consider the $s$-derivative of
$$\chi_t(y\exp(sW)a_tz_0).$$
We rewrite this as
$$\chi_t(ya_t\exp(s\Ad(a_t)^{-1}W)z_0)$$
and apply the projection along $\hf$. It follows that
the derivative can be rewritten as
$$d/ds|_{s=0} \,\chi_t(ya_t
\exp(s\pi_\vf\Ad(a_t)^{-1}W)z_0)$$
and then finally also as
$$
d/ds|_{s=0} \,\chi_t(y\exp(s\Ad(a_t)\pi_\vf\Ad(a_t)^{-1}W)a_tz_0).$$
Note that $\Ad(a_t)\pi_\vf\Ad(a_t)^{-1}W\in\vf$.
We conclude that the derivative is a linear combination
of derivatives of $\psi$ on $V$,
with coefficients that are smooth functions on $V$.
Furthermore, it follows from Lemma \ref{M_t}
that the coefficients are bounded for $t\leq 0$. 
As before we conclude
$L(W)\chi_t\in L^p(Z)$ for all $t\leq 0$,
with exponentially decaying $p$-norms. It follows
that $L(W)\chi\in L^p(Z)$. 

By repeating the argument
for higher derivatives we finally see that $\chi\in L^p(Z)^\infty$.
This concludes the proof of Proposition \ref{non-red}.
\end{proof}

\begin{rmk} It follows from the proof of the proposition
that 
\begin{equation}\label{limit zero} 
\lim_{t\to -\infty} \mathbf{v}_B(a_t\cdot z_0) =0\, .\end{equation}
In fact if we apply the invariant Sobolev Lemma \ref{Blemma} to the function 
$\chi$ with $p=1$ we get 
$$n \leq \chi(a_{-n}\cdot z_0) \leq C_B v_B(a_{-n}\cdot z_0)^{-1} 
\|\chi\|_{1, 2\dim G} \ \quad (n\in \N)\, .$$
Thus, for a constant $C>0$,  
$$\mathbf{v}_B(a_{-n}\cdot z_0) \leq \frac Cn \qquad (n\in \N)\, .$$
The assertion (\ref{limit zero}) follows from the fact that 
the equivalence class of $v_B$ is independent of the choice of the ball
$B$ and that $a_{t} a_{[t]}^{-1}\in B'$ for all $t\in \R$ and a certain ball 
$B'$. 
\end{rmk}

The following general result was used above.

\begin{lemma} \label{grad-lemma} Let $\uf=\bigoplus_{j>0} \uf^j$ be a positively graded nilpotent 
Lie algebra and $\hf<\uf$ a subalgebra. Let $\uf_0\subset \uf$
be a graded vector complement to $\hf$ which is constructed as follows:
If $\hf=\uf$, then $\uf_0=\{0\}$. Otherwise we choose  a vector, say $Y_1$, in $\uf^{j_1}\setminus \hf$,
which  largest possible $j_1$. If $\hf+\R Y_1 \subsetneq \uf$
we choose  $Y_2\in \uf^{j_2}\setminus (\hf+\R Y_1)$
with largest possible $j_2$. We 
continue this procedure until $Y_1,Y_2,\dots$
span a complementary subspace. This subspace we denote $\uf_0$. 

Let $U$ be a simply connected Lie group with Lie algebra $\uf$ and $H<U$ the connected subgroup associated to $\hf$.
Then the map 
$$ \uf_0 \to U/H, \ \ X\mapsto \exp(X)H$$
is a diffeomorphism. 
\end{lemma}

\begin{proof} By induction on $\dim \uf$. The one-dimensional case is trivial.
Let $0\neq Y$ be an element in $\uf$ of top degree and 
note that $Y$ is central. If $Y_1\in\uf^{\operatorname{top}}$ we choose $Y=Y_1$. 
Otherwise $\uf^{\operatorname{top}}\subset\hf$,
and we choose $Y$ arbitrarily. We consider the graded Lie algebra $\tilde \uf :=\uf/ \R Y$ 
and the subalgebra $\tilde\hf=\hf/\R Y$. The assertion now follows easily by 
applying the induction hypothesis to this pair. 
Note that if $Y=Y_1$ then
$$\exp(t_1Y_1+\dots+ t_m Y_m)=\exp(t_1Y_1)\exp(t_2Y_2+\dots+ t_m Y_m)$$
since $Y$ is central.
\end{proof}

\subsection{Final remarks}

\par {\bf 1.} 
We did not address here the case where 
$G$ is not reductive. One might expect in general for $G/H$ unimodular and algebraic 
that $Z$ has VAI if and only if the
nilradical of $H$ is contained in the nilradical of~$G$.

\par {\bf 2.} 
The following may be an alternative approach to
Theorem \ref{th=1}.
To be more specific, assume $Z=G/H$ to be unimodular,
algebraic and quasi-affine. 
Under these assumptions we expect that there
is a rational $G$-module $V$, and an embedding $Z\to V$ such that
the invariant measure $\mu_Z$, via pull-back, defines a tempered
distribution on $V$. Note that if $Z$ is of reductive type, then
there exists a $V$ such that the image of $Z\to V$ is closed, and hence
$\mu_Z$ defines a tempered distribution on $V$. If $Z$ is not of reductive
type, then by Matsushima's criterion (\cite{BH-C}, Thm.~3.5)
all images $Z\to V$ are non-closed 
and the expected embedding would imply that VAI does not hold.
This is supported by a result in \cite{RR},
which asserts that
for a reductive group $G$ and $X\in \gf:=\mathrm{Lie}(G)$ the invariant measure
on the adjoint orbit $Z:=\Ad(G)(X)\subset \gf$ defines a tempered distribution
on $\gf$.
Various  particular results in the theory of prehomogeneous vector spaces
provide additional support (see~\cite{BR}).

\end{document}